\documentclass{amsart} \usepackage{latexsym,amsxtra,amscd,ifthen} \usepackage{amsfonts} \usepackage{verbatim}
\usepackage{amsmath} \usepackage{amsthm} \usepackage{amssymb}

\input xy \xyoption{matrix} \xyoption{arrow}\xyoption{frame} 
 
 \newcommand{\edge}{\ar@{-}}

\numberwithin{equation}{section}

\theoremstyle{plain} \newtheorem{theorem}{Theorem}[section] \newtheorem{lemma}[theorem]{Lemma}
\newtheorem{proposition}[theorem]{Proposition} \newtheorem{corollary}[theorem]{Corollary}

\theoremstyle{definition} \newtheorem{definition}[theorem]{Definition} \newtheorem{example}[theorem]{Example}

\newtheorem{claim}[theorem]{Claim}
\newtheorem{remark}[theorem]{Remark}  \newtheorem*{remark*}{Remark} \newtheorem{question}[theorem]{Question}

\newcommand{\gr}{\operatorname{gr}}

\newcommand{\gnoc}{\mathrel{{\lower.2ex\hbox{$\backsim$}}\llap{\raise.45ex\hbox{=}}}}

\begin{document}

\title[] {A note on the order of the antipode of a pointed Hopf algebra}

\author[Paul Gilmartin]{P. Gilmartin} \address{School of Mathematics and Statistics\\ University of Glasgow\\ Glasgow G12 8QW\\
Scotland.} \email{p.gilmartin.1@research.gla.ac.uk}


\thanks{Some of these results will form part of the authorÕs PhD thesis at the University of
Glasgow, supported by a Scholarship of the Carnegie Trust. The author would like to thank Ken Brown
for very helpful discussions.}


\maketitle

\begin{abstract} 
Let $k$ be a field and let $H$ denote a pointed Hopf $k$-algebra with antipode $S$. We are interested in determining the order of $S$. Building on the work done by Taft and Wilson in \cite{Taft2}, we define an invariant for $H$, denoted $m_{H}$, and prove that the value of this invariant is connected to the order of $S$. In the case where $\operatorname{char}k=0$, it is shown that if $S$ has finite order then it is either the identity or has order $2m_{H}$. If in addition $H$ is assumed to be coradically graded, it is shown that the order of $S$ is finite if and only if $m_{H}$ is finite. We also consider the case where $\operatorname{char}k=p>0$, generalising the results of \cite{Taft2} to the infinite-dimensional setting.

\end{abstract}

\section{Introduction}\label{intro}


In this paper we are interested in determining the order of the antipode of a pointed Hopf algebra over an arbitrary field $k$. We suspect that the results of this paper are well-known to experts, but their proofs are apparently lacking in the literature.\par
For a pointed Hopf $k$-algebra $H$, we introduce an invariant, denoted $m_{H}$, which is, in a sense which we shall make precise, a  measure of the extent to which a group-like element $x\in H$ commutes with any $h\in H$ such that
$\Delta(h)=h\otimes x+1\otimes h$. Whilst in general $m_{H}$ can take values in $\mathbb{Z}^{\geq 0}\cup\{\infty\}$, the condition that $m_{H}$ is finite is valid in a variety of natural settings, for example whenever $G(H)$ is finite or central in $H$ (see Proposition \ref{central}). We record our main results below (where part (3) appears as \textrm{Theorem \ref{0main}}  and part (4) appears as Corollary \ref{pmain2}). These results connect  the order of the antipode of a pointed Hopf algebra $H$ to the value of $m_{H}$ in both the cases of zero and positive characteristic. For the relevant definitions see $\S\ref{prelim0}$.

\begin{theorem}\label{anti main}
Let $k$ be a field. Suppose $H$ is a pointed Hopf $k$-algebra. Let $\{H_{n}\}_{n\geq 0}$ denote the coradical filtration of $H$ and let $\gr{H}$ denote the associated graded pointed Hopf algebra with respect to the coradical filtration. Let $S$ and $\overline{S}$ denote the antipode of $H$ and $\gr{H}$ respectively.
\begin{enumerate}
\item{
If $m_{H}=\infty$ then $|S|=\infty$.}
\item{
(Taft, Wilson) If $m_{H}<\infty$ then $(S^{2m_{H}}-\operatorname{id})(H_{n})\subseteq H_{n-1}$ for $n\geq 1$.}
\item{
Suppose $\operatorname{char}k=0$.  If $|S|<\infty$ then either $S=\operatorname{id}$ or 
$$|S| = 2m_{H}=2m_{\gr{H}}=|\overline{S}|.$$
}
\item{
Suppose $\operatorname{char}k=p>0$. If $m_{H}<\infty$ and $H = k\langle H_{n}\rangle$ for some $n\geq 0$, then $|S|$ divides  $2m_{H}p^{l}$, where $l\in\mathbb{N}$ is such that $p^{l}\geq n \geq p^{l-1}.$
}
\end{enumerate}

\end{theorem}

Theorem \ref{anti main}(2), from which everything else quickly follows, has exactly the same proof as part (1) of the following result of Taft and Wilson from 1974, and so we credit it to them. Parts (3) and (4) of Theorem \ref{anti main} should be compared to the analogous results for $H$ finite-dimensional as presented in part (2) of the result below. It will be clear from the definition that $m_{H}$ divides the exponent of $G(H)$ whenever the exponent is finite.

 \begin{theorem}(Taft, Wilson, \cite{Taft2})\label{TW}
 Let $k$ be a field and let $H$ be a pointed Hopf $k$-algebra. Assume that $G(H)$ has finite exponent $e$.
 \begin{enumerate}
 \item{(\cite[\textrm{Proposition 3, Proposition 4}]{Taft2}) For $n\geq 1$, $(S^{2e}-\operatorname{id})(H_{n})\subseteq H_{n-1}$.}
 \item{(\cite[\textrm{Corollary 6}]{Taft2})
 Assume that $H$ is finite-dimensional and that $H=H_{n}$ for some $n\geq 0$. If $\operatorname{char}k =0$ and $S$ has finite order, then $|S|$ divides $2e$. If $\operatorname{char}k=p>0$, $S^{2ep^{m}}=\operatorname{Id}$, where $p^{m}\geq n > p^{m-1}$.
}
\end{enumerate}
\end{theorem}

\begin{remark}
It was subsequently proved by Radford, \cite{Radford2}, that the order of the antipode of a finite-dimensional Hopf algebra is always finite, allowing us to drop the assumption that $S$ has finite order in the finite dimensional setting of Theorem \ref{TW}(2).
\end{remark}

As noted in Remark \ref{connected}, the converse of  Theorem \ref{anti main}(1) is in general not true. However, as an almost immediate consequence of 
of Theorem \ref{anti main}(2), in the case where $H$ is known to be coradically graded (see Definition \ref{corad def}), the condition that $m_{H}<\infty$ is equivalent to the condition that $S$ has finite order. That is, we deduce the following, which appears later as Proposition \ref{coradprop}.

\begin{proposition}\label{garfunkel}
Let $k$ be a field and let $H$ be a pointed coradically graded Hopf $k$-algebra.
\begin{enumerate}
\item{$|S|=\infty$ if and only if $m_{H}=\infty$.}
\item{If $m_{H}<\infty$, $S=\operatorname{id}$ or $|S|=2m_{H}$.}
\end{enumerate}
\end{proposition}

\begin{remark}
\begin{enumerate}
\item{
In positive characteristic, there exist examples of pointed Hopf algebras where the antipode has order strictly less than the bound obtained in Theorem \ref{anti main}(4). Take, for example, a field $k$ such that $\operatorname{char}k=p>0$ and let $H$ be a pointed coradically Hopf graded Hopf $k$-algebra with $m_{H}<\infty$ and $H\neq H_{0}$ (see Example \ref{borel} for an explicit example of such a Hopf algebra). By Proposition \ref{garfunkel}, $|S|=2m_{H}$, which is strictly less than the bound obtained in Theorem \ref{anti main}.

On the other hand, there exist examples of pointed Hopf algebras over fields of positive characteristic where the bound obtained in Theorem \ref{anti main} (4) is actually attained. In Example \ref{bad pos} we give an example, originally due to Taft and Wilson, \cite{Taft}, of a finite dimensional connected Hopf algebra $R$ over a field of characteristic $p\geq 3$ with $R=R_{2}$, $m_{R}=1$  and an antipode of order $2p$.
}
\item{
 Note that for an arbitrary pointed Hopf algebra $H$, $m_{H}$ will be in general strictly less than the exponent of $G(H)$ - in Example \ref{borel} we see that, if  $q$ is a primitive $n^{\operatorname{th}}$ root of unity, the pointed coradically Hopf graded Hopf algebra  $H=U_{q}(\mathfrak{b}^{+})$ has the property that $m_{H}=n$ and that $G(H)$ has  infinite exponent.}
 \end{enumerate}
\end{remark}

\section{Preliminaries}\label{prelim0}

Throughout $k$ will denote an arbitrary field (unless otherwise stated). For a Hopf $k$-algebra $H$ the usual notation $\Delta, \epsilon$ and $S$ will denote the coproduct, counit and antipode respectively. By the \emph{order} of the antipode $S$, which we shall denote by $|S|$, we mean the minimal $n$ such that $S^{n}=\operatorname{id}$, the identity map of $H$.\par
The \emph{coradical} filtration of  a Hopf algebra $H$ is denoted $\{H_{n}\}_{n=0}^{\infty}$, where $H_{0}$ is the coradical of $H$ (that is, the sum of  its simple subcoalgebras), and we define inductively, for all $n\geq 1$,
$$H_{n}:=\Delta^{-1}(H\otimes H_{i-1} + H_{0}\otimes H).$$
For a Hopf algebra $H$, an element $g\in H$ is said to be \emph{group-like} if $\Delta(g)=g\otimes g$. It is a simple exercise to prove that the set of all group-like elements of $H$ forms a group, which we shall denote $G(H)$. A Hopf algebra $H$ is said to be \emph{pointed} if $H_{0}=kG(H)$ (or, equivalently, if each simple subcoalgebra is one-dimensional). As proved in \cite[Lemma 5.2.8]{Mont}, for example, if $H$ is pointed then the coradical filtration $\{H_{n}\}_{n=0}^{\infty}$ is in fact a Hopf algebra filtration of $H$ and hence the associated graded space with respect to the coradical filtration, which we denote by $\gr{H}$, inherits the structure of a pointed Hopf algebra from $H$.

\subsection{Defining $m_{H}$}

The aim of this section is to define, for any pointed Hopf algebra $H$, the invariant $m_{H}$ which appears in Theorem \ref{anti main}. We begin by recalling the definition of a skew-primitive element of a pointed Hopf algebra. 

\begin{definition}\label{bigspag}
Suppose $H$ is a pointed Hopf algebra. For any $x, y\in G(H)$, define the space of \emph{$(x, y)$-skew-primitive elements} of $H$, 
\[
P_{x,y}(H):=\left\{h\in H : \Delta(h)=h\otimes x+y\otimes h \right\}.
\]
\end{definition}

\begin{remark}\label{note 1}
As noted in the opening remarks of \cite[\textrm{\S 5.4}]{Mont}, if $H$ is a pointed Hopf algebra and $x, y\in G(H)$, then $P_{x,y}(H)\cap H_{0} = k(x - y)$. For each such pair $x, y\in G(H)$, let $P_{x, y}(H)'$ denote a subspace such that
$$P_{x, y}(H) = k(x-y)\oplus P_{x,y}(H)'.$$
\end{remark}

The following result, which appears as stated below as \cite[\textrm{Theorem 5.4.1}]{Mont}, but is originally due to Taft and Wilson, \cite{Taft2}, is the crux of the proof of our main result, Theorem \ref{anti main}.

\begin{theorem}\label{diss}
Let $H$ be a pointed Hopf algebra. Then
$$H_{1} = kG(H)\oplus(\oplus_{x, y\in G} P_{x,y}(H)').$$
\end{theorem}

We also require the following well-known and easy lemma.

\begin{lemma}\label{bigspace}
Let $H$ be a pointed Hopf algebra, let $x\in G(H)$ and let $\langle{x}\rangle$ denote the subgroup of $G(H)$ generated by $x$.
\begin{enumerate}
\item{$P_{x,1}(H)$ is an $\langle x\rangle$-invariant subspace of $H$, where $x$ acts by conjugation.}
\item{
$P_{x,1}(H)\subseteq \operatorname{ker}\epsilon.$}
\end{enumerate}
\end{lemma}
\begin{proof}
\begin{enumerate}
\item{Let $h\in P_{x,1}(H)$. Since $\Delta$ is an algebra homomorphism,
$$
\begin{aligned}
\Delta(xhx^{-1})&=(x\otimes x)(h\otimes x +1\otimes h)(x^{-1}\otimes x^{-1})\\
&=xhx^{-1}\otimes x+1\otimes xhx^{-1}
\end{aligned}
$$
hence $xhx^{-1}\in P_{x,1}(H)$.}
\item{Let $h\in P_{x,1}(H)$. By the counit axiom of the coproduct, $h\epsilon(x)+\epsilon(h)=h$. Since $x$ is group-like, $\epsilon(x)=1$. The result follows.
}
\end{enumerate}
\end{proof}


In light of Lemma \ref{bigspace} (1), we make the following definition.
 
\begin{definition}
Let $H$ be a pointed Hopf algebra. For any $x\in G(H)$, define
$$a_{x} = |\langle x\rangle :C_{\langle x\rangle}(P_{x,1}(H))|$$
where $C_{\langle x\rangle}(P_{x,1}(H))$ denotes the centraliser of $P_{x,1}(H)$ in $\langle x\rangle$ under the conjugation action by $\langle x\rangle$.
\end{definition}



\begin{definition}
For a pointed Hopf algebra $H$, define
$$m_{H}:=\operatorname{lcm}\{a_{x}:x\in G\}.$$
\end{definition}

We record a couple of simple observations about $m_{H}$.

\begin{proposition}\label{central}
Let $H$ be a pointed Hopf algebra.
\begin{enumerate}
\item{If $G(H)$ is central in $H$ then $m_{H}=1$.}
\item{
Suppose $G(H)$ is finite. Then $m_{H}$ is finite and divides the exponent of the group $G(H)$.}
\item{
If $H=H_{0}=kG(H)$ then $m_{H}=1$.}
\end{enumerate}
\end{proposition}
\begin{proof}
Parts (1) and (2) are immediate from the way we defined $m_{H}$. For part (3), let $x\in G(H)$. If $H=kG(H)$, then as noted in Remark \ref{note 1},  $P_{x, 1}(H) = k(x - 1)$. It follows that $a_{x}=1$ and hence $m_{H}=1$.
\end{proof}

For an arbitrary pointed Hopf algebra $H$, $m_{H}$ can take values in $\mathbb{Z}^{\geq 0}\cup\{\infty\}$ and will be in general strictly less than the exponent of the group $G(H)$, as shown by the following example.

\begin{example}\label{borel}
Let $k$ be a field and let $0, 1\neq q\in k$. Set $H={U}_{q}(\mathfrak{b}^{+})$, the quantised enveloping algebra of the positive two dimensional  Borel. This is defined as the algebra generated by the letters $E, K$ and $K^{-1}$, subject to the relations $KK^{-1}=1=K^{-1}K$ and 
\[
KE=qEK.
\]
Then, as proved in \cite[\textrm{I.3.4}]{BrownGoodearl}, for example, $H$ becomes a pointed Hopf algebra, with coproduct $\Delta:H\rightarrow H\otimes H$ and antipode $S:H\rightarrow H$ defined on generators as follows
\[
\Delta(E)=E\otimes 1+K\otimes E,\quad  \Delta(K)=K\otimes K
\]
\[
S(K)=K^{-1},\quad  S(E)=-K^{-1}E.
\]
Set $E':= EK^{-1}\in P_{K^{-1},1}(H)$. An elementary calculation shows that $S^{2n}(E')=q^{-n}E'$  and $K^{n}E'=q^{n}E'K^{n}$ for all $n\geq 1$. Since $K, K^{-1}$ and $E'$ form a set of generators of $H$, it is clear that the value of $m_{H}$ depends only on the action of $K$ on $E'$. Thus
\begin{enumerate}
\item{
If $q$ is an $n^{th}$ primitive root of unity for some $1\leq n <\infty$, $G(H)$ has infinite exponent, $m_{H}=n$ and $|S|=2n$.}
\item{
If $q$ is $\emph{not}$ a root of unity, then $G(H)$ has infinite exponent, $m_{H}=\infty$ and $|S|=\infty$.}
\end{enumerate}

\end{example}

\section{Preliminary computations}\label{m sec}

The main results of this section, Proposition \ref{151} and Proposition \ref{153}, along with their proofs, are almost identical to \cite[\textrm{Proposition 3, Theorem 5}]{Taft2}, the only difference being that, for a pointed Hopf algebra $H$,  we state our results in terms of $m_{H}$, instead of the exponent of $G(H)$, and do not restrict ourselves to stating the result for $H$ being  finite-dimensional only.
\\

The results of this section are valid over any field.

\begin{lemma}\label{antipodeform1}
Let $H$ be a pointed Hopf algebra. Let $x\in G(H)$. Then, for $h\in P_{x,1}(H)$,
$S(h)=-hx^{-1}$.
\end{lemma}
\begin{proof}
Let $x\in G(H)$ and choose $h\in P_{x,1}(H)$, so that $\Delta(h)=h\otimes x+1\otimes h$. By Lemma \ref{bigspace}(2), $\epsilon(h)=0$. By the counit axiom of the antipode,
\[
S(h)x+ h=0.
\]
That is, $S(h)=-hx^{-1}$.
\end{proof}

The following lemma is valid over any field.

\begin{lemma}\label{claim1}
Let $H$ be a pointed Hopf algebra. Let  $x\in G(H)$ and let $h\in P_{x,1}(H)$.
\begin{enumerate}
\item{For any $m\geq1$,
$S^{2m}(h)=x^{m}hx^{-m}.$}
\item{If in addition $m_{H}<\infty$,
$S^{2m_{H}}(h)=h.$
}
\end{enumerate}
\end{lemma}

\begin{proof}

Fix $h\in  P_{x,1}(H)$. By Lemma \ref{antipodeform1}, $S(h)=-hx^{-1}$. This gives
\begin{equation}\label{squared}
S^{2}(h)=-S(x^{-1})S(h)=xhx^{-1}.
\end{equation}
Proceeding inductively, for any $m\geq 1$,
\[
S^{2m}(h)=x^{m}hx^{-m}.
\]
It is then clear from the definition of $m_{H}$ that if $m_{H}<\infty$, $S^{2m_{H}}(h)=h$.

\end{proof}

\begin{proposition}\label{151}
Let $H$ be a pointed Hopf algebra with $m_{H}<\infty$. Then $$(S^{2m_{H}} - \operatorname{Id})(H_{1}) = 0.$$
\end{proposition}
\begin{proof}
Let $h\in H_{1}$. If $h\in H_{0}$ then $S^{2}(h)=h$ since $H_{0}=kG(H)$ and $S^{2}|_{kG(H)}=\operatorname{id}$. Thus  by Theorem \ref{diss} and linearity, we can without loss of generality assume that $h\in P_{x, y}(H)$ for some $x, y\in G$. Using the fact that $\Delta$ is an algebra homomorphism, an elementary calculation then shows that $hy^{-1}\in P_{xy^{-1}, 1}(H)$. Then, using Lemma \ref{claim1}(2) and the fact that $S^{2m_{H}}$ is an algebra morphism,
$$hy^{-1}= S^{2m_{H}}(hy^{-1}) = S^{2m_{H}}(h)S^{2m_{H}}(y^{-1})=S^{2m_{H}}(h)y^{-1},$$ 
giving $S^{2m_{H}}(h) = h$, as required.

\end{proof}

Next we shall prove a sufficient condition for the antipode of a pointed Hopf algebra to have infinite order. Before we do this, we need the following well-known lemma.

\begin{lemma}\label{always even}
Let $H$ be a Hopf algebra with $|S|<\infty$.  Then either $S=\operatorname{id}$ or $|S|=2k$ for some $k\in\mathbb{N}$.
\end{lemma}
\begin{proof}
Suppose $S\neq{id}$. If $H$ is commutative, $|S|=2$ by \cite[\textrm{Corollary 1.5.12}]{Mont}, so without loss of generality assume that $H$ is noncommutative. Choose $x, y\in H$ such that $xy\neq yx$. Let $|S|=m<\infty$. Suppose $m$ is odd - write $m=2q+1$ for some $q\geq 1$. Since $S$ is an anti-algebra morphism, so too is $S^{2q+1}$, hence
$$xy= S^{2q+1}(xy)=S^{2q+1}(y)S^{2q+1}(x)=yx.$$
This contradicts the assumption that $x$ and $y$ do not commute, thus $S$ must have even order.

\end{proof}

\begin{proposition}\label{infinite order}
Let $H$ be a pointed Hopf algebra with $m_{H} = \infty$. Then $|S|=\infty$.
\end{proposition}
\begin{proof}
Suppose $|S|<\infty$. By Lemma \ref{always even}, we can write $|S| = 2t$ for some $t\in\mathbb{N}$. Then by  Lemma \ref{claim1}(1), for any $x\in G(H)$ and $h\in P_{x,1}(H)$,
$$h = S^{2t}(h) = x^{t}hx^{-t}.$$
If $m_{H}=\infty$, there exists some $y\in G(H)$ and $f\in P_{y,1}(H)$ such that $y^{t}fy^{-t}\neq f$. Thus it must be that $m_{H}<\infty$.
\end{proof}

Following the arguments of \cite{Taft2}, the following result allows us to extend Proposition \ref{151} to higher terms in the coradical filtration.

\begin{proposition}\label{152}
Let $H$ be a pointed Hopf algebra,  let $i\geq 1$ and let $\psi:H\rightarrow H$ be a coalgebra homomorphism. Suppose $(\psi  - \operatorname{id})(H_{j})\subseteq H_{j-1}$ for all $0\leq j\leq i$. Then $(\psi - \operatorname{id})(H_{i+1})\subseteq H_{i}$.
\end{proposition}
\begin{proof}
This is \cite[\textrm{Proposition 4}]{Taft2}.
\end{proof}

\begin{proposition}\label{153}
Let $H$ be a pointed Hopf algebra with $m_{H}<\infty$. Then, for any $n\geq 1$,
\begin{enumerate}
\item{$(S^{2m_{H}} - \operatorname{id})(H_{n})\subseteq H_{n-1}$.
}
\item{
$(S^{2m_{H}}-\operatorname{id})^{n}(H_{n}) = 0$.}
\end{enumerate}
\begin{proof}
This is immediate from Proposition \ref{151}, Proposition \ref{152} and the fact that $S^{2m_{H}}$ is a coalgebra morphism.
\end{proof}
\end{proposition}

\subsection{Coradically graded Hopf algebras}

\begin{definition}\label{corad def}
Let $H$ be a Hopf algebra. We say a family of subspaces $\{H(n)\}_{n\geq 0}$ is a \emph{Hopf algebra grading} of $H$  if $\{H(n)\}_{n\geq 0}$ is both an algebra and coalgebra grading with the additional property that $S(H(n))\subseteq H(n)$ for all $n\geq 0$. If in addition we have, for each $n\geq 0$, that
$$H_{n}=\bigoplus_{i=0}^{n}H(i)$$
we say $H$ is a \emph{coradically graded Hopf algebra}.
\end{definition}

\begin{proposition}\label{coradprop}
Let $H$ be a pointed coradically graded Hopf algebra. Then 
\begin{enumerate}
\item{$m_{H}=\infty$ if and only if $|S|=\infty$.}
\item{If $m_{H}<\infty$, $S=\operatorname{id}$ or $|S|=2m_{H}$.}
\end{enumerate}
\end{proposition}
\begin{proof}
If $m_{H}=\infty$, $|S|=\infty$ by Proposition \ref{infinite order}. For the converse,  let $H=\bigoplus_{i=0}^{\infty}H(i)$ be a pointed coradically  graded Hopf algebra, so that, for $n\geq 1$, $H_{n}=\bigoplus_{i=0}^{n}H(i)$, and let $m_{H}<\infty$.  Fix $n\geq 1$. By Proposition \ref{153}, it follows that
$$(S^{2m_{H}}-\operatorname{id})(H(n))\subseteq \bigoplus_{j=0}^{n-1}H(j-1).$$
However, since $\{H(n)\}_{n}$ is a Hopf grading, $(S^{2m_{H}}-\operatorname{id})(H(n))\subseteq H(n)$ for each $n\geq 0$. It must therefore be that $(S^{2m_{H}}-\operatorname{id})(H(n)) = 0$ for $n\geq 0$, hence $S^{2m_{H}}=\operatorname{id}$. To complete the proof, it suffices to prove that for any $q<m_{H}$, there exists $h\in H$ such that $S^{2q}(h)\neq h$, since Lemma \ref{always even} guarantees that the order of the antipode is always either $1$ or divisible by $2$. Suppose for a contradiction that there exists some $q<m_{H}$ such that $|S| = 2q$. By Lemma \ref{claim1}(1), for any $x\in G(H)$, $h\in P_{x,1}(H)$,
$$h  = S^{2q}(h) = x^{q}hx^{-q}.$$
However, since $q<m_{H}$, by definition there exists some $y\in G$ and $f\in P_{y,1}(H)$ such that
$$f \neq y^{q}fy^{-q},$$
a contradiction. This completes the proof.

\end{proof}

Let $H$ be a pointed Hopf algebra. As is mentioned at the beginning of $\S\ref{prelim0}$, the associated graded space with respect to the coradical filtration of $H$, $\gr{H}$, inherits the structure of a pointed Hopf algebra from $H$. It is well-known, and proved in \cite[Proposition 4.4.15]{Radford}, for example, that, with respect to the Hopf structure inherited from $H$, $\gr{H}$ becomes a pointed coradically graded Hopf algebra. The following is thus an immediate corollary of Proposition \ref{coradprop}.

\begin{corollary}\label{assoco}
Let $H$ be a pointed Hopf algebra with $m_{H}<\infty$. Then $m_{\gr H}<\infty$ and either $S=\operatorname{id}$ or  $|S_{\gr H}|=2m_{\gr H}$.
\end{corollary}

\section{The antipode in characteristic zero}

We now consider what happens when we work over a field of characteristic $0$.

\begin{proposition}\label{0main}
 Let $H$ be a pointed Hopf $k$-algebra. Suppose $\operatorname{char}k=0$. If $m_{H}=\infty$ then $|S|=\infty$. If $m_{H}<\infty$ then either $|S|$ divides ${2m_{H}}$, or there exists $h\in H$ such that the orbit of $S$ on $h$ is infinite. In particular, either $|S|$ divides $2m_{H}$ or $|S|=\infty$.
 
\end{proposition}
\begin{proof}




We can, without loss of generality, assume that $m_{H}<\infty$, since otherwise Proposition \ref{infinite order} guarantees that $|S| =\infty$. Suppose $S^{2m_{H}}\neq \operatorname{id}$, and choose $n$ minimal such that $h\in H_{n}$ and $S^{2m_{H}}(h)\neq h$. We shall prove that $S^{2l}(h)\neq h$ for any $l\geq 1$. By Proposition \ref{153}(1), $S^{2m_{H}}(h) = h+r$ for some $r\in H_{n-1}$. By the choice of $h$, we can assume $r\neq 0$.

\begin{claim}\label{claim2}
Retain the above notation. For $t\geq 1$, $S^{2m_{H}t}(h)=h+tr$.
\end{claim}
\begin{proof}\textit{of Claim \ref{claim2}}:
We proceed by induction on $t\geq 1$. The $t=1$ case was Proposition \ref{153}(1). Fix $t\geq 1$. Then, by Proposition \ref{153}(1),
\begin{equation}\label{0}
S^{2(t+1)m_{H}}(h)=S^{2tm_{H}}S^{2m_{H}}(h)=S^{2tm_{H}}(h)+S^{2tm_{H}}(r).
\end{equation}
By the minimality of $n$, $S^{2tm_{H}}(r)=r$. By the inductive hypothesis,  equation ($\ref{0}$) becomes
\[
S^{2(t+1)m_{H}}(h)=(h+tr)+r=h+(t+1)r
\]
proving the claim by induction.

\end{proof}

So, $S^{2m_{H}t}(h)=h+tr$ for all $t\geq 1$. Since $\operatorname{char}k=0$, this implies $S^{2m_{H}t}(h)\neq h$ for all $t\geq 1$. Thus $|S^{2m_{H}}|=\infty$, and so $|S|=\infty$. In particular, if $i$ and $j$ are distinct integers, then $S^{i}(h)\neq S^{j}(h)$, since otherwise $S^{2m_{H}(i-j)}(h)=h$, contradicting the above. This completes the proof.


\end{proof}

\begin{theorem}\label{normal main}
Let $H$ be a pointed Hopf $k$-algebra, where  $\operatorname{char}k = 0$. If $|S|<\infty$ then either $|S| = 2m_{H}$ or $S=\operatorname{id}$.
\end{theorem}
\begin{proof}
Suppose $S\neq \operatorname{id}$. By Proposition \ref{0main} and Lemma \ref{always even}, it suffices to show that if $|S|<\infty$, then, for any $q<m_{H}$, there exists $h\in H$ such that $S^{2q}(h)\neq h$.

Suppose for a contradiction that there exists some $q<m_{H}$ such that $|S| = 2q$. By Lemma \ref{claim1} (1), for any $x\in G$, $h\in P_{x,1}(H)$,
$$h  = S^{2m_{H}}(h) = x^{q}hx^{-q}.$$
However, since $q<m_{H}$, by definition there exists some $y\in G$ and $f\in P_{y,1}(H)$ such that
$$f \neq y^{q}fy^{-q},$$
a contradiction. This completes the proof.
\end{proof}

Combining Corollary \ref{assoco} and Theorem \ref{normal main}, we prove the following result which connects the order of the antipode of a pointed Hopf algebra $H$ to the order of the antipode of the associated graded Hopf algebra, $\gr{H}$. The proof is essentially exactly the same as the proof of Proposition \ref{0main}.

\begin{theorem}\label{final main}
Let $H$ be a pointed Hopf $k$-algebra, where $\operatorname{char}k=0$. If $|S_{H}|<\infty$ then  $|S_{H}|=|S_{\gr{H}}|.$
\end{theorem}
\begin{proof}
Suppose $S\neq{id}$. If $|S_{H}|<\infty$ then $m_{H}<\infty$ by Proposition \ref{infinite order}. Let $S$ denote the antipode of $H$, let $\overline{S}$ denote the antipode of $\gr{H}$ and let $l=m_{\gr{H}}$. By Corollary \ref{assoco} then it follows that $l<\infty$ and $|\overline{S}|=2l$. Let $n\geq 0$, $h\in H_{n}$ and let $\overline{h}=h+H_{n-1}\in \gr{H}$. Then $\overline{S}^{2l}(\overline{h})=\overline{h}$ and so
$$S^{2l}(h) = h+r$$
for some $r\in H_{n-1}$. Then, by exactly the same proof as the one given for Claim \ref{claim2} above, we get that
$$S^{2lt}(h)=h+tr$$
for each $t\geq 1$. If $r\neq 0$ then, since $\operatorname{char}{k}=0$, $|S^{2l}|=\infty$ and hence $|S|=\infty$. However we assumed $|S|<\infty$, so we must have $r=0$. Thus we have $S^{2l}=\operatorname{id}$ and so $|S|$ divides $2l$. However, $|S|=2m_{H}$ by Theorem $\ref{normal main}$, which implies that $m_{H}$ divides $l$. Clearly $l=m_{\gr H} \leq m_{H}$ in general, and so it must be that $m_{\gr{H}} = m_{H}$. The result follows.

\end{proof}

\begin{corollary}\label{summary}
Let $H$ be a pointed Hopf $k$-algebra. Assume $\operatorname{char}k=0$ . Then
\begin{enumerate}
\item{If $G(H)$ is central in $H$, then either $|S|=\infty, |S|= 2$ or $S=\operatorname{id}$.}
\item{
If $H$ is a connected Hopf algebra (i.e. $G(H)=\{1\}$), then either $|S|=\infty, |S|=2$ or $S=\operatorname{id}$.}
\end{enumerate}
\end{corollary}
\begin{proof}
Clearly if $G(H)$ is central in $H$ then $m_{H}=1$, so (1) follows from Proposition \ref{0main}. Part (2) is immediate from (1).

\end{proof}

\begin{remark}\label{connected}
Note that in Theorem \ref{normal main} all alternatives can occur, even in the case where $H$ is connected, where $m_{H}=1$ always. It is noted in \cite[\S{3.5}]{QHom} that the connected Hopf algebra $B(\lambda)$ introduced by Zhuang in \cite[\S{7}]{Zhuang} has an antipode of infinite order.
\end{remark}

\section{The antipode in positive characteristic}

In \cite[\textrm{Corollary 6}]{Taft2}, a bound is obtained for the order of the antipode of a finite dimensional Hopf algebra $H$ over a field of positive characteristic in terms of the exponent of $G(H)$. It turns out that the proof of that result is equally valid if we drop the assumption that $H$ is finite dimensional, and also that, in light of the results of $\S\ref{m sec}$, the result can be restated in terms of $m_{H}$, rather than the exponent of $G(H)$. As such, we again credit the results of this section to Taft and Wilson.

\begin{proposition}\label{pmain}
Let $H$ be a pointed Hopf $k$-algebra with $m_{H}<\infty$. Let $n\geq1$ and $h\in H_{n}$. Choose $l\in\mathbb{N}$ such that $p^{l}\geq n \geq p^{l-1}$. If $\operatorname{char}k=p>0$, $S^{2m_{H}p^{l}}(h)=h$.
\end{proposition}
\begin{proof}
Let $n\geq1$ and $h\in H_{n}$. Choose $l\in\mathbb{N}$ such that $p^{l}\geq n \geq p^{l-1}$. By Proposition Lemma \ref{153}(2),
$$0 = (S^{2m_{H}}-\operatorname{id})^{p^{l}}(h)=S^{2m_{H}p^{l}}(h) - h.$$
where the final equality follows from the binomial theorem in characteristic $p>0$, which works here since $S$ and $\operatorname{id}$ commute in $\operatorname{End}_{k}(H)$. The result follows.
\end{proof}

The following corollary is immediate.

\begin{corollary}\label{pmain2}
Let $H$ be a pointed Hopf $k$-algebra with $m_{H}<\infty$. Assume $\operatorname{char}k=p>0$.
\begin{enumerate}
\item{
Suppose that $H$ can be generated as an algebra by $H_{n}$ for some $0\leq n <\infty$. Choose $l\in\mathbb{N}$ such that $p^{l}\geq n \geq p^{l-1}$. Then $|S|$ divides $2m_{H}p^{l}$.}
\item{
If $H$ is affine (that is, finitely generated as an algebra), then $|S|$ divides $2m_{H}p^{n}$ for some $0\leq n <\infty$. In particular, $|S|<\infty$.}
\end{enumerate}

\end{corollary}

\begin{corollary}\label{finally}
Let $H$ be an affine pointed Hopf $k$-algebra, where $\operatorname{char}k=p>0$.
\begin{enumerate}
\item{
Suppose $G(H)$ is central in $H$. Then there exists some $n\geq 0$ such that $|S|$ divides $2p^{n}$.}
\item{
Suppose $H$ is connected. Then there exists some $n\geq 0$ such that $|S|$ divides $2p^{n}$.}
\end{enumerate}
\end{corollary}
\begin{proof}
Part (1) follows immediately from the Corollary \ref{pmain2} and fact that if $G(H)$ is central in $H$ then $m_{H}=1$. Part (2) is a special case of (1).
\end{proof}

As shown by the following example, originally due to Taft and Wilson, \cite{Taft}, the bound on the order of the antipode obtained in Corollary $\ref{pmain2}$ is in general not attained, even in the finite-dimensional connected case.

\begin{example}\label{bad pos}
Let $k$ be a field with $\operatorname{char}k=p\geq 3$. Let $R$ be the algebra with generators $X, Y$ and $Z$ subject to the following relations.
\[
[X, Y] = X \quad [Y, Z] = -Z, \quad [X,Z] =\frac{1}{2}X^{2},
\]
\[
X^{p}=0, \quad Y^{p}=Y, \quad Z^{p}=0.
\]
It is proved in \cite{Taft} that $R$ is a connected Hopf algebra of vector space dimension $p^{3}$ with coproduct, counit and antipode defined on generators as follows:
\[
\epsilon(X)=0, \quad \Delta(X)=1\otimes X+X\otimes 1,
\]
\[
\epsilon(Y)=0, \quad \Delta(Y)=1\otimes Y+Y\otimes 1,
\]
\[
\epsilon(Z)=0, \quad \Delta(Z)=1\otimes Z+X\otimes Y+Z\otimes 1,
\]
\[
S(X)=-X,\quad S(Y)=-Y, \quad S(Z)=-Z+XY.
\]
Since $R$ is connected, $m_{R}=1$. Morover, notice that $X, Y\in R_{1}$ and $Z\in R_{2}$, so $R$ can be generated in at least coradical degree $2$. If $R$ was generated in coradical degree one, it would be cocommutative, since  $R$ being connected implies $R_{1}=k\oplus P(R)$, \cite[\textrm{Lemma 5.3.2}]{Mont}. Since $p\geq 3$, the bound on the order of the antipode of $R$ as determined by
Corollary \ref{pmain2} is therefore $2p$. A simple calculation  yields the identity
$$S^{2t}(Z) = Z -tX$$
for any $t\geq 1$. In particular, $S^{2p}(Z)=Z$. Since $S$ is an anti-algebra morphism, it follows  that $|S|=2p$.
\end{example}  

\begin{example}
For an example of a pointed Hopf $k$-algebra $H$ over a field of positive characteristic which has an antipode of infinite order, see Example \ref{borel}:  when $q$ is not a root of unity, the Hopf algebra $H=U_{q}(\mathfrak{b}^{+})$ has an antipode of infinite order over \emph{any} field. 
\end{example}

We know of no example of a connected Hopf algebra $H$ in positive characteristic  with an antipode of infinite order. This prompts the following question.

\begin{question}
Suppose $H$ is a connected Hopf $k$-algebra, where $\operatorname{char}k=p>0$. Does the antipode of $H$ always have finite order?
\end{question}

By Corollary \ref{finally}(2), an example which gives a negative answer to the above question would be necessarily non-affine. More generally, we could ask the following.

\begin{question}
Suppose $H$ is a pointed Hopf $k$-algebra, where $\operatorname{char}k=p>0$ and $m_{H}<\infty$. Does the antipode of $H$ always have finite order?
\end{question}

\end{document}